\documentclass[11pt]{ip-journal} 

\usepackage[applemac]{inputenc}
\usepackage{amsmath, amsthm, amssymb}
\usepackage[all]{xy}
\usepackage{marginnote}
\usepackage{color}

\theoremstyle{plain}
\newtheorem{theorem}{Theorem}[section]

\newtheorem{proposition}[theorem]{Proposition}

\newtheorem{lemma}[theorem]{Lemma}

\newtheorem{question}[theorem]{Question}
\theoremstyle{definition}
\newtheorem{definition}[theorem]{Definition}

\theoremstyle{remark}
\newtheorem{remark}[theorem]{Remark}

\DeclareMathOperator{\HSC}{HSC}
\DeclareMathOperator{\Ric}{Ric}

\begin{document}
\bibliographystyle{amsalpha}

\title[Holomorphic sectional curvature and canonical bundle]{Quasi-negative holomorphic sectional curvature and positivity of the canonical bundle} 

\author{Simone Diverio \and Stefano Trapani}
\address{Simone Diverio \\ Istituto \lq\lq Guido Castelnuovo\rq\rq{} \\ SAPIENZA Università di Roma \\ Piazzale Aldo Moro, 5 \\ I-00185 Roma. \\}
\email{diverio@mat.uniroma1.it} 
\address{Stefano Trapani \\ Università di Roma \lq\lq Tor Vergata\rq\rq{} \\ Via della Ricerca Scientifica 1 \\ I-00133 Roma. \\}
\email{trapani@mat.uniroma2.it}

\thanks{The first–named author was partially supported by the ANR Programme: Défi de tous les savoirs (DS10) 2015, \lq\lq GRACK\rq\rq{}, Project ID: ANR-15-CE40-0003ANR, and Défi de tous les savoirs (DS10) 2016, \lq\lq FOLIAGE\rq\rq{}, Project ID: ANR-16-CE40-0008.}
\keywords{Holomorphic sectional curvature, Monge--Amp\`ere equation, canonical bundle}
\subjclass{Primary: 32Q15; Secondary: 32Q05.}

\begin{abstract}
We show that if a compact complex manifold admits a K\"ahler metric whose holomorphic sectional curvature is everywhere non positive and strictly negative in at least one point, then its canonical bundle is positive. This answers in the affirmative to a question first asked by S.-T. Yau.
\end{abstract}

\maketitle

\section{Introduction}

Let $(X,\omega)$ be a Kähler manifold and let $\Theta(T_X,\omega)$ be its Chern curvature. The holomorphic sectional curvature of $\omega$ in the direction given by a tangent vector $v\in T_{X,x}\setminus\{0\}$ is defined by
$$
\HSC_\omega(x,[v])=\frac 1{||v||_\omega^4}\bigl\langle\Theta_x(T_X,\omega)\cdot v,v\bigr\rangle_\omega(v,\bar v).
$$
In the above formula, $\Theta(T_X,\omega)$ firstly acts as an endomorphism of the holomorphic tangent space and then, once contracted again with $v$ using the hermitian product defined by $\omega$, eats the pair $(v,\bar v)$ as a $(1,1)$-form. Despite it is well-known that the holomorphic sectional curvature completely determines the Chern curvature tensor, it is not \textsl{a priori} clear wether and how its sign propagates and determines the signs of the other curvature tensors. An average argument anyway shows that there is a direct link between the sign of the holomorphic sectional curvature and the scalar curvature, \textsl{i.e.} the trace of the Ricci curvature, namely the positivity (resp. the negativity) of the former implies the positivity (resp. the negativity) of the latter. 

However, it was conjectured by S.-T. Yau that a compact Kähler manifold $(X,\omega)$ with negative holomorphic sectional curvature should aways admit (a possibly different) Kähler metric $\omega'$ with negative Ricci curvature. Then, $K_X$ should be ample and, in particular, $X$ should be projective. This conjecture has been proved only very recently in the projective case by Wu and Yau in \cite{WY16}. After this major breakthrough, the result was extended by using essentially the same techniques to the Kähler case in \cite{TY15}. Before these achievements, only some cases were known under some extra conditions. For instance, this was proven in \cite{HLW10} supposing the abundance conjecture to hold true (which is the case indeed in dimension less than or equal to three).

This circle of ideas fits also in the general conjectural picture of Kobayashi hyperbolicity of compact Kähler manifolds. Namely, it was conjectured by Kobayashi himself that a compact hyperbolic Kähler manifold should have positive canonical bundle. Now, if a compact complex manifold admits a hermitian metric of negative holomorphic sectional curvature, then it is well-known that the manifold in question is Kobayashi hyperbolic (the converse does not hold in general, see \cite[Theorem 8.2]{Dem97} for a very interesting class of projective examples). Thus, negativity of the holomorphic sectional curvature is a strong way to have hyperbolicity and the result of Wu--Yau can also be seen as a weak confirmation of the Kobayashi conjecture for projective manifolds (and Wu--Yau coupled with Tosatti--Yang for the Kähler case).

Now, what about compact Kähler manifolds with merely non positive holomorphic sectional curvature? Such manifolds surely have nef canonical bundle thanks to \cite{TY15} (in the projective case this is a well-known consequence of Mori's theorem, since they do not admit any rational curve, see next section for more details). Anyway, such a condition is not strong enough in order to obtain positivity of the canonical bundle, as flat complex tori show. A less obvious but still easy counterexample is given by the product of a flat torus and, say, a compact Riemann surface of genus greater than or equal to two endowed with its Poincaré metric. In this example, over each point there are some directions with strictly negative holomorphic sectional curvature but always some flat directions, too. We refer the reader to the paper \cite{HLW14,HLWZ17} for some nice results about the merely non positive case, including an interesting structure theorem for such manifolds, built upon the so-called nef fibration.

The above discussion can be seen as a motivation for the following (standard, indeed) definition.

\begin{definition}
The holomorphic sectional curvature is said to be \emph{quasi-negative} if $\HSC_\omega\le 0$ and moreover there exists at least one point $x\in X$ such that $\HSC_\omega(x,[v])< 0$ for every $v\in T_{X,x}\setminus\{0\}$.
\end{definition}

The quasi-negative situation was also already considered by S.-T. Yau, who asked whether the same conclusions of the negative case still hold. Here comes our main contribution which is a positive answer to Yau's question.

\begin{theorem}\label{thm:main}
Let $(X,\omega)$ be a connected compact Kähler manifold. Suppose that the holomorphic sectional curvature of $\omega$ is quasi-negative. Then, $K_X$ is ample. In particular, $X$ is projective.
\end{theorem}

A particular case of this theorem was already proved in \cite{WWY12} under the additional assumption that the Picard group of $X$ is infinite cyclic, and in \cite{HLW14} under the additional assumption that $X$ is a projective surface. Our approach is essentially borrowed from the work of S.-T. Yau and his collaborators on the subject, namely a Monge--Ampère type equation and Yau's refined Schwarz Lemma, together with some ingredients from pluripotential theory.

Let us record that shortly after the first version of this paper was put on the ArXiv, Wu and Yau in \cite{WY16b} gave a unified proof of the results contained here, as well as in \cite{WY16,TY15}. Also, R. Nomura \cite{Nom16} gave an alternative proof in the strictly negative case by means of the Kähler--Ricci flow.

Let us finish the introduction with the following question. We saw that, by \cite{TY15}, a compact Kähler manifold $X$ with a Kähler metric $\omega$ whose holomorphic sectional curvature is non positive has nef canonical bundle. On the other hand, by the celebrated Abundance Conjecture, such a manifold should have semiample canonical bundle. Now, if $K_X$ is semiample, the cohomology class $-c_1(X)$ has a semipositive smooth representative, hence Yau's solution of the Calabi conjecture implies the existence of a (possibly different) Kähler metric $\omega'$ with non positive Ricci curvature. It is thus legitimate to ask:

\begin{question}
Let $(X,\omega)$ be a compact Kähler manifold such that $\HSC_\omega\le 0$. Does there exist a (possibly different) Kähler metric $\omega'$ on $X$ such that $\Ric(\omega')\le 0$? 
\end{question}


\subsubsection*{Acknowledgments} We would like to warmly thank V. Tosatti and H. Guenancia for interesting and generous exchanges. 
\section{Reduction to the key inequality}

Let $(X,\omega)$ be a $n$-dimensional compact Kähler manifold such that $\HSC_\omega\le 0$. Then, it is classically known (see for instance \cite[Corollary 2]{Roy80}) that $X$ cannot contain any (possibly singular) rational curve, \textsl{i.e.} it does not admit any non constant map $\mathbb P^1\to X$. Now, if $X$ is projective, Mori's theorem immediately gives us that $K_X$ must be nef. If $X$ is merely supposed to be Kähler, then the nefness of $K_X$ still holds true and is a direct consequence of the non positivity of the holomorphic sectional curvature, but this is a much more recent result \cite[Theorem 1.1]{TY15}. 

Now, suppose that one can show under the quasi-negativity assumption of the holomorphic sectional curvature that 
\begin{equation}\label{eqn:kxbig}
c_1(K_X)^n>0.
\end{equation}
Then, by \cite[Theorem 0.5]{DP04}, we deduce that $K_X$ is big. In particular, carrying a big line bundle, $X$ is Moishezon. Since $X$ is Kähler and Moishezon, by Moishezon's theorem $X$ is projective. But then, the following lemma implies that $K_X$ is ample.

\begin{lemma}[Exercise 8, page 219 of \cite{Deb01}]\label{lem:ratcurv}
Let $X$ be a smooth projective variety of general type which contains no rational curves. Then, $K_X$ is ample.
\end{lemma}

Here is a proof, for the sake of completeness.

\begin{proof}
Since there are no rational curves on $X$, Mori's theorem implies as above that $K_X$ is nef. Since $K_X$ is big and nef, the Base Point Free theorem tells us that $K_X$ is semi-ample. If $K_X$ were not ample, then the morphism defined by (some multiple of) $K_X$ would be birational but not an isomorphism. In particular, there would exist an irreducible curve $C\subset X$ contracted by this morphism. Therefore, $K_X\cdot C=0$. Now, take any very ample divisor $H$. For any $\varepsilon>0$ rational and small enough, $K_X-\varepsilon H$ remains big and thus some large positive multiple, say $m(K_X-\varepsilon H)$, of $K_X-\varepsilon H$ is linearly equivalent to an effective divisor $D$. Set $\Delta=\varepsilon' D$, where $\varepsilon'>0$ is a rational number.  We have:
$$
\begin{aligned}
(K_X+\Delta)\cdot C & =\varepsilon'\,D\cdot C \\
& =\varepsilon'm(K_X-\varepsilon H)\cdot C \\
& =-\varepsilon\varepsilon'm\,H\cdot C<0.
\end{aligned}
$$
Finally, if $\varepsilon'$ is small enough, then $(X,\Delta)$ is a klt pair. Thus, the (logarithmic version of the) Cone Theorem would give the existence of an extremal ray generated by the class of a rational curve in $X$, contradiction.
\end{proof}

\begin{remark}
The same conclusion can be directly obtained by means of \cite[Theorem 1.1]{Tak08}. This theorem states, among other things, that the non-ample locus of the canonical divisor of a smooth projective variety of general type is uniruled. In particular, if there are no rational curves, the non-ample locus must be empty and thus $K_X$ is ample.
\end{remark}

It is thus sufficient to prove inequality (\ref{eqn:kxbig}). Since $K_X$ is nef, for any $\varepsilon>0$, the cohomology class $[\varepsilon\omega-\Ric(\omega)]=\varepsilon[\omega]+c_1(K_X)$ is a Kähler class. By \cite[Proposition 8]{WY16}, for every $\varepsilon>0$, there exists a smooth function $u_\varepsilon$ which solves the following Monge--Ampère equation:

\begin{equation}\label{eqn:ma}
\begin{cases}
\bigl(\varepsilon\omega-\Ric(\omega)+i\partial\bar\partial u_\varepsilon\bigr)^n=e^{u_\varepsilon}\omega^n,\\ 
\omega_\varepsilon:=\varepsilon\omega-\Ric(\omega)+i\partial\bar\partial u_\varepsilon>0.
\end{cases}
\end{equation}

Moreover, again by \cite[Proposition 8]{WY16}, there exists a constant $C>0$ which only depends on $\omega$ and $n=\dim X$, such that
\begin{equation}\label{unifbdd}
\sup_X u_\varepsilon <C.
\end{equation}
Now, 
$$
\begin{aligned}
\int_X e^{u_\varepsilon}\,\omega^n & =\int_X\omega_\varepsilon^n \\
&= \bigl(\varepsilon[\omega]+c_1(K_X)\bigr)^n \\
& =c_1(K_X)^n+\sum_{j=0}^{n-1}{n \choose j}\varepsilon^{n-j}[\omega]^{n-j}\cdot c_1(K_X)^j.
\end{aligned}
$$
Therefore, 
$$
\lim_{\varepsilon\to 0^+}\int_X e^{u_\varepsilon}\,\omega^n=c_1(K_X)^n,
$$
and what we have to show is that 
\begin{equation}\label{eqn:mainineq}
\lim_{\varepsilon\to 0^+}\int_X e^{u_\varepsilon}\,\omega^n>0.
\end{equation}
The next section will be entirely devoted to the proof of inequality (\ref{eqn:mainineq}), which in the sequel will be referred to as \lq\lq key inequality\rq\rq{}.

\section{Proof of the key inequality}

The first observation is that the functions $u_\varepsilon$ we are considering are all $\omega'$-plurisubharmonic for some fixed Kähler form $\omega'$ and $\varepsilon>0$ small enough. For, let $\ell>0$ be such that $\ell\omega-\Ric(\omega)$ is positive and call $\omega'=\ell\omega-\Ric(\omega)$. Thus, for all $0<\varepsilon<\ell$, one has 
$$
0<\varepsilon\omega-\Ric(\omega)+i\partial\bar\partial u_\varepsilon<\ell\omega-\Ric(\omega)+i\partial\bar\partial u_\varepsilon=\omega'+i\partial\bar\partial u_\varepsilon.
$$
Moreover, the $u_\varepsilon$'s are uniformly bounded from above thanks to (\ref{unifbdd}). 

\begin{lemma}[See for instance {\cite[Proposition 2.6]{GZ05}} and {\cite[Proposition 4.8]{GZ17}}]\label{hartogs}
Either $\{u_\varepsilon\}$ converges uniformly to $-\infty$ on $X$ or it is relatively compact in $L^1(X)$.
\end{lemma}

Since this property will be crucial for our approach, following the referee's suggestion, we explain how to get the global case needed here from the standard local case (see for exemple \cite[Theorem 1.46]{GZ17}).

\begin{proof}
Suppose that $\{u_\varepsilon\}$ does not converge uniformly to $-\infty$. Then, there exist a subsequence $\{u_{\varepsilon_k}\}$ and a sequence of points $\{x_k\}\subset X$ such that $\{\sup_X u_{\varepsilon_k}=u_{\varepsilon_k}(x_k)\}_k$ is bounded from above and below, and $\{x_k\}$ converges to some point $x_0\in X$. 

Now, fix a finite atlas $\{V_i\}$ for $X$ consisting of open coordinate charts which are relatively compact in some other bigger local charts. Let $x_0$ be say in $V_0$. By the local version of the lemma, up to subsequences, $u_{\varepsilon_k}$ converges in $L^1(V_0)$. Let $V_1$ be such that $V_0\cap V_1\ne\emptyset$. Since $u_{\varepsilon_k}$ converges in $L^1(V_0\cap V_1)$, up to a further subsequence, again by the local version of the lemma, $u_{\varepsilon_k}$ converges in $L^1(V_1)$. Iterating the reasoning we get the convergence in $L^1(V_i)$, for all $i$, and hence in $L^1(X)$.
\end{proof}

Now, suppose for a moment that we are in the second case of Lemma \ref{hartogs}. Then, there exists a subsequence $\{u_{\varepsilon_k}\}$ converging in $L^1(X)$ and moreover the limit coincides \textsl{a.e.} with a uniquely determined $\omega'$-pluri\-sub\-harmonic function $u$. Up to pass to a further subsequence, we can also suppose that $u_{\varepsilon_k}$ converges pointwise \textsl{a.e.} to $u$. But then, $e^{u_{\varepsilon_k}}\to e^u$ pointwise \textsl{a.e.} on $X$. On the other hand, we have $e^{u_{\varepsilon_k}}\le e^C$ so that, by dominated convergence, we also have $L^1(X)$-convergence and
$$
\lim_{k\to\infty}\int_X e^{u_{\varepsilon_k}}\,\omega^n=\int_X e^u\omega^n>0.
$$
The upshot is that what we need to prove is that $\{u_\varepsilon\}$ does not converge uniformly to $-\infty$ on $X$. From now on, we shall suppose by contradiction that 
$$
\sup_X u_\varepsilon\to-\infty.
$$
Now, as in \cite{WY16}, consider the smooth positive function $S_\varepsilon$ on $X$ defined by
$$
\omega\wedge\omega_\varepsilon^{n-1}=\frac{S_\varepsilon}n\,\omega_\varepsilon^{n}.
$$
Now, define $T_\varepsilon$ to be $\log S_\varepsilon$. In other words, $T_\varepsilon$ is the logarithm of the trace of $\omega$ with respect to $\omega_\varepsilon$.

\begin{lemma}\label{lem:sge-m}
The function $T_\varepsilon$ satisfies the following inequality:
$$
T_\varepsilon\ge -\frac{u_\varepsilon}n.
$$
In particular, if $\{u_\varepsilon\}$ converges uniformly to $-\infty$ on $X$, then $T_\varepsilon$ converges uniformly to $+\infty$ on $X$.
\end{lemma}

\begin{proof}
Let $0<\lambda_{1}\le\cdots\le\lambda_n$ be the eigenvalues of $\omega_\varepsilon$ with respect to $\omega$. Then,
$$
e^{T_\varepsilon}=\operatorname{tr}_{\omega_\varepsilon}\omega=\frac 1{\lambda_1}+\cdots+\frac 1{\lambda_n}\ge\frac 1{\lambda_1}.
$$
Thus, $e^{-T_\varepsilon}\le\lambda_1$ so that $e^{-nT_\varepsilon}\le(\lambda_1)^n\le\lambda_1\cdots\lambda_n$.
But, $e^{u_\varepsilon}\omega^n=\omega_\varepsilon^n=\lambda_1\cdots\lambda_n\,\omega^n$, and so we get $
e^{-nT_\varepsilon}\le e^{u_\varepsilon}$, or, in other words,
$$
T_\varepsilon\ge-\frac{u_\varepsilon}n.
$$
\end{proof}

Next, since we do not dispose of a negative constant uniform upper bound for $\HSC_\omega$, we are naturally led to consider the following continuous function on $X$:
$$
\begin{aligned}
\kappa\colon & X\to\mathbb R \\
& x\mapsto -\max_{v\in T_{X,x}\setminus\{0\}}\HSC_\omega(x,[v]).
\end{aligned}
$$
The quasi-negativity of the holomorphic sectional curvature of Theorem \ref{thm:main} translates in $\kappa\ge 0$ and $\kappa(x_0)>0$ for some $x_0\in X$.

By \cite[Proposition 9]{WY16}, for every $\varepsilon>0$ we have the following crucial inequality which makes the holomorphic sectional curvature enter into the problem:
\begin{equation}\label{eq:maindiffineq}
\Delta_{\omega_\varepsilon}T_\varepsilon(x)\ge\biggl(\frac{n+1}{2n}\kappa(x)+\frac\varepsilon n\biggr)e^{T_\varepsilon(x)}-1.
\end{equation}
Let us set $M(x)=(n+1)\kappa(x)/2n$ and
$$
\overline M_\varepsilon=\frac{\int_X M\,\omega_\varepsilon^n}{\int_X\omega_\varepsilon^n}=
\frac{\int_X Me^{u_\varepsilon}\,\omega^n}{\int_X e^{u_\varepsilon}\,\omega^n}.
$$

\begin{lemma}\label{lem:Mepsilon}
There exist a sequence $\{\varepsilon_k\}$ of positive real numbers converging to zero as $k$ goes to infinity such that
$$
\overline M_{\varepsilon_k}\to\overline M_0>0.
$$
\end{lemma}

\begin{proof}
Let us define 
$$
v_\varepsilon(x):=u_\varepsilon(x)-\max_X u_\varepsilon.
$$
Then,  $\{v_\varepsilon\}$ is a family of $\omega'$-plurisubharmonic functions on $X$ such that for every $\varepsilon>0$ one has
$$
\max_X v_\varepsilon = 0.
$$
Therefore, the functions $v_{\varepsilon_k}$ cannot converge uniformly to $- \infty$, and we see as above that, up to extracting a further subsequence, the function $e^{v_{\varepsilon_k}}$ tends to $e^{v}$ in $L^1(X)$ and \textsl{a.e.} pointwise for some quasi-plurisubharmonic function $v$. Next, since we have that $e^{v_{\varepsilon_k}}\le 1$ and $Me^{v_{\varepsilon_k}}\le\max_X M$, by dominated convergence we obtain that both $e^{v_{\varepsilon_k}}$ and $Me^{v_{\varepsilon_k}}$ converge in $L^1(X)$ respectively to $e^v$ and $Me^v$, and moreover 
$$
\int_X Me^{v_{\varepsilon_k}}\,\omega^n\to\int_X Me^v\,\omega^n>0,\quad 
\int_X e^{v_{\varepsilon_k}}\,\omega^n\to\int_X e^v\,\omega^n>0,
$$ 
since $M>0$ on some open set and $e^v> 0$ almost everywhere. To conclude, it suffices to observe that
$$
\frac{\int_X Me^{u_\varepsilon}\,\omega^n}{\int_X e^{u_\varepsilon}\,\omega^n}=\frac{\int_X Me^{v_\varepsilon}\,\omega^n}{\int_X e^{v_\varepsilon}\,\omega^n}.
$$
\end{proof}

Now, for every $\varepsilon>0$, we consider the differential equation on $X$
\begin{equation}\label{eq:kw}
\Delta_{\omega_\varepsilon}\varphi=Me^\varphi-1.
\end{equation}

\begin{proposition}\label{prop:subsuper}
Suppose that $\varphi_+,\varphi_-\in C^2(X)$ are such that 
$$
\Delta_{\omega_\varepsilon}\varphi_-\ge Me^{\varphi_-}-1,\quad\text{and}\quad
\Delta_{\omega_\varepsilon}\varphi_+\le Me^{\varphi_+}-1.
$$
Then, $\varphi_-\le\varphi_+$.
\end{proposition}

\begin{proof}
Let $\Omega\subset X$ be the set of points $x\in X$ such that $\varphi_-(x)>\varphi_+(x)$. First, observe that $\Omega$ cannot be the whole of $X$. Indeed, by subtracting the two differential inequalities in the statement we have that
$$
\Delta_{\omega_\varepsilon}(\varphi_--\varphi_+)\ge M(e^{\varphi_-}-e^{\varphi_+}),
$$
so that we would obtain $\Delta_{\omega_\varepsilon}(\varphi_--\varphi_+)\ge 0$ everywhere on $X$. But then, $\varphi_--\varphi_+$ would be constant and hence we would get
$$
M(e^{\varphi_-}-e^{\varphi_+})\le 0.
$$
But this is impossible, since there is at least one point of $X$ where $M$ is strictly positive. Therefore, $\Omega$ is a proper open subset of $X$. In this open subset one has by definition that $\varphi_-(x)>\varphi_+(x)$, which implies that $\Delta_{\omega_\varepsilon}(\varphi_--\varphi_+)\ge 0$, and moreover $\varphi_-=\varphi_+$ on the boundary $\partial\Omega$. By the maximum principle, we thus get a contradiction.
\end{proof}

Next, inspired by \cite{KW74}, we want to construct a supersolution of (\ref{eq:kw}). We first settle a regularity issue.

\begin{lemma}
The function $M\colon X\to\mathbb R$ is Lipschitz.
\end{lemma}

\begin{proof}
It suffices of course to prove that $\kappa\colon X\to\mathbb R$ is Lipschitz. Suppose the contrary. Then, there exists a sequence $(p_i,q_i)\in X\times X$ such that 
$$
\lim_{i\to +\infty}\frac{|\kappa(p_i)-\kappa(q_i)|}{d_X(p_i,q_i)}=+\infty,
$$
where $d_X$ is the distance on $X$ induced by $\omega$. Without loss of generality, up to extract a subsequence,  we can suppose that both $\{p_i\}$ and $\{q_i\}$ converge to the same point $x_0\in X$. Fix a normal, geodesically convex, relatively compact coordinate neighborhood $V$ of $x_0$. The distance $d_X|_V$ is thus equivalent to the flat Euclidean metric on $V$. 
By smoothly trivializing the tangent bundle over $V$ via a $\omega$-unitary local frame, we are led to the simpler situation where we consider a smooth function on the product $\mathbb B^n \times \mathbb C^n,$ where $\mathbb B^n$ is the open unit  ball in $\mathbb C^n$, and take the maximum on the unit sphere in the second set of variables. Such a maximum is clearly a Lipschitz function on compact sets of $\mathbb B^n$.
\end{proof}

\begin{remark}
V. Tosatti kindly communicated to us the following alternative way to proceed. Instead of taking care of the regularity properties of $M$, one may construct an auxiliary function $\widetilde M$ which is smooth and nonnegative on $X$, strictly positive at $x_0$ and bounded above by $M$. Such a function is easily constructed, since $M$ is continuous. 

Now, in (\ref{eq:maindiffineq}), replace $M(x)=\frac{n+1}{2n}\kappa(x)$ by $\widetilde M(x)$. The inequality remains then true and the arguments above and right here below goes through as well without any issue of regularity for the RHS of the equation and then with now smooth $f_\varepsilon$'s.
\end{remark}

Now, since by construction 
$$
\int_X(M- \overline M_\varepsilon)\,\frac{\omega_\varepsilon^n}{n!}=0
$$ 
and $M$ is Lipschitz, we can find for each $\varepsilon>0$ a unique $C^2(X)$-solution $f_\varepsilon$ of the differential equation
$$
\Delta_{\omega_\varepsilon}f_\varepsilon= M- \overline M_\varepsilon,
$$
such that $\inf_X f_\varepsilon=0$. Our next goal is to find uniform real constants $A,B$ such that $A\,f_\varepsilon+B$ is a supersolution of (\ref{eq:kw}). Therefore, we would like to find $A,B$ such that
$$
\Delta_{\omega_\varepsilon}(A\,f_\varepsilon+B)=A\,\Delta_{\omega_\varepsilon}f_\varepsilon\le M\,e^{A\,f_\varepsilon+B}-1,
$$
that is
$$
1-A\overline M_\varepsilon\le M\bigl(e^{A\,f_\varepsilon+B}-A\bigr).
$$
Since $\inf_X f_\varepsilon =0$, then 
$$
e^{A\,f_\varepsilon+B}-A\ge e^B-A.
$$ 
We thus choose $A,B$ to be any real numbers respectively such that $A$ is greater than $1/\overline M_0$ and $B$ is greater than $\log A$. But then, along the subsequence extracted in Lemma \ref{lem:Mepsilon}, for all $k$ large enough, we have
$$
1-A\overline M_{\varepsilon_k}<0,
$$
and 
$$
M\bigl(e^{A\,f_\varepsilon+B}-A\bigr)\ge M\bigl(e^B-A\bigr)\ge 0
$$
holds for any $\varepsilon>0$.
In particular, for all $k$ large enough, $A\,f_{\varepsilon_k}+B$ is a supersolution of (\ref{eq:kw}) with $\varepsilon=\varepsilon_k$. But
$$
\Delta_{\omega_\varepsilon}T_\varepsilon\ge\biggl(\frac{n+1}{2n}\kappa+\frac\varepsilon n\biggr)e^{T_\varepsilon}-1\ge Me^{T_\varepsilon}-1,
$$
so that $T_\varepsilon$ is a subsolution of (\ref{eq:kw}) for all $\varepsilon>0$. By Proposition \ref{prop:subsuper}, for each $k$ large enough, we obtain
$$
T_{\varepsilon_k}\le A\,f_{\varepsilon_k}+B
$$
on $X$.

To conclude the proof of the key inequality, recall that we are assuming by contradiction that 
$$
\lim_{\varepsilon\to 0}\sup_X u_\varepsilon\to -\infty,
$$
and that this would imply that $\lim_{\varepsilon\to 0}\inf_X T_\varepsilon\to +\infty$, thanks to Lemma \ref{lem:sge-m}. But then,
$$
\lim_{\varepsilon\to 0}\inf_X T_\varepsilon=\lim_{k\to +\infty} \inf_X T_{\varepsilon_k}\le\lim_{k\to +\infty} \inf_X(A\,f_{\varepsilon_k}+B)=B,
$$
because of our choice $\inf_X f_\varepsilon =0$. This is absurd.

\begin{remark}
Note that if the holomorphic sectional curvature had been strictly negative, then much more simply we could have chosen as a supersolution a constant (large enough) function. This would have immediately given a uniform upper bound for $T_\varepsilon$, as in \cite{WY16,TY15}.
\end{remark}

\begin{remark}
The following elegant and somehow quicker way to conclude, which bypasses the use of sub and supersolutions, has been kindly communicated to us by H. Guenancia shortly after a first version of the present paper appeared on the ArXiv. Start from the inequality 
$$
\Delta_{\omega_\varepsilon}T_\varepsilon\ge Me^{T_\varepsilon}-1.
$$
Now, integrate over $X$ using the volume form associated to $\omega_\varepsilon$, to get
$$
0=\int_X\Delta_{\omega_\varepsilon}T_\varepsilon\,\omega_\varepsilon^n\ge\int_X\bigl(Me^{T_\varepsilon}-1\bigr)\,\omega_\varepsilon^n.
$$
We obtain therefore the following integral inequality:
$$
\int_X Me^{T_\varepsilon}e^{u_\varepsilon}\,\omega^n\le\int_X e^{u_\varepsilon}\,\omega^n,
$$
and setting $v_\varepsilon=u_\varepsilon-\sup_X u_\varepsilon$ one has
$$
\int_X Me^{T_\varepsilon}e^{v_\varepsilon}\,\omega^n\le\int_X e^{v_\varepsilon}\,\omega^n.
$$
Finally, if we define $C_\varepsilon=\inf_X e^{-u_\varepsilon/n}$, we have that $e^{T_\varepsilon}> C_\varepsilon$, and
$$
C_\varepsilon\int_X Me^{v_\varepsilon}\,\omega^n\le\int_X e^{v_\varepsilon}\,\omega^n.
$$
But then, as in Lemma \ref{lem:Mepsilon}, we can extract a subsequence of $\{v_\varepsilon\}$ with limit $v$ which makes the two integral involved converge to a finite non zero limit. This is a contradiction since $C_\varepsilon\to +\infty$.
\end{remark}

\begin{remark}
Now that we know that $K_X$ is ample, we also know by the classical work of Aubin and Yau that there exists a unique smooth function $u$ on $X$ such that 
$$
\bigl(-\Ric(\omega) + i \partial \bar\partial u\bigr)^n = e^u\,\omega^n
$$ 
and 
$$
-\Ric(\omega) + i \partial \bar{\partial} u > 0.
$$

We claim, as in \cite{WY16}, that $u_{\varepsilon}$ converges to this solution $u$ in each $C^{k, \alpha}$-topology. Indeed, by standard arguments in the theory of Monge--Ampère equations it is sufficient to show hat there exist a uniform $C^0$-estimate for $u_{\varepsilon}$ (see for example the proof of \cite[Theorem 5.1]{BEGZ10}, or \cite[ pp. 360 and 363]{Yau78}). We already know that there exists a uniform upper bound, so we only need a uniform lower bound.

Now, the function $u$ is a subsolution of each of the equations 
$$
\bigl(-\Ric(\omega)  + \varepsilon \omega + i \partial \bar{\partial} u_{\varepsilon}\bigr)^n = e^{u_{\varepsilon}}\,\omega^n,
$$ 
then by \cite[Theorem 2.18]{EGZ11} we have that $u_{\varepsilon } \ge \inf u$ for all $0 < \varepsilon < 1$.
\end{remark}

\bibliography{bibliography}{}

\providecommand{\bysame}{\leavevmode\hbox to3em{\hrulefill}\thinspace}
\providecommand{\MR}{\relax\ifhmode\unskip\space\fi MR }
\providecommand{\MRhref}[2]{%
  \href{http://www.ams.org/mathscinet-getitem?mr=#1}{#2}
}
\providecommand{\href}[2]{#2}
\begin{thebibliography}{{Nom}16}

\bibitem[BEGZ10]{BEGZ10}
S{\'e}bastien Boucksom, Philippe Eyssidieux, Vincent Guedj, and Ahmed Zeriahi,
  \emph{Monge-{A}mp\`ere equations in big cohomology classes}, Acta Math.
  \textbf{205} (2010), no.~2, 199--262. \MR{2746347}

\bibitem[Deb01]{Deb01}
Olivier Debarre, \emph{Higher-dimensional algebraic geometry}, Universitext,
  Springer-Verlag, New York, 2001. \MR{1841091}

\bibitem[Dem97]{Dem97}
Jean-Pierre Demailly, \emph{Algebraic criteria for {K}obayashi hyperbolic
  projective varieties and jet differentials}, Algebraic geometry---{S}anta
  {C}ruz 1995, Proc. Sympos. Pure Math., vol.~62, Amer. Math. Soc., Providence,
  RI, 1997, pp.~285--360. \MR{1492539}

\bibitem[DP04]{DP04}
Jean-Pierre Demailly and Mihai P{\u{a}}un, \emph{Numerical characterization of
  the {K}\"ahler cone of a compact {K}\"ahler manifold}, Ann. of Math. (2)
  \textbf{159} (2004), no.~3, 1247--1274. \MR{2113021}

\bibitem[EGZ11]{EGZ11}
Philippe Eyssidieux, Vincent Guedj, and Ahmed Zeriahi, \emph{Viscosity
  solutions to degenerate complex {M}onge-{A}mp\`ere equations}, Comm. Pure
  Appl. Math. \textbf{64} (2011), no.~8, 1059--1094. \MR{2839271}

\bibitem[GZ05]{GZ05}
Vincent Guedj and Ahmed Zeriahi, \emph{Intrinsic capacities on compact
  {K}\"ahler manifolds}, J. Geom. Anal. \textbf{15} (2005), no.~4, 607--639.
  \MR{2203165}

\bibitem[GZ17]{GZ17}
\bysame, \emph{Degenerate complex {M}onge-{A}mp\`ere equations}, EMS Tracts in
  Mathematics, vol.~26, European Mathematical Society (EMS), Z\"urich, 2017.
  \MR{3617346}

\bibitem[HLW10]{HLW10}
Gordon Heier, Steven S.~Y. Lu, and Bun Wong, \emph{On the canonical line bundle
  and negative holomorphic sectional curvature}, Math. Res. Lett. \textbf{17}
  (2010), no.~6, 1101--1110. \MR{2729634}

\bibitem[HLW16]{HLW14}
\bysame, \emph{K\"ahler manifolds of semi-negative holomorphic sectional
  curvature}, J. Differential Geom. \textbf{104} (2016), no.~3, 419--441.
  \MR{3568627}

\bibitem[HLWZ17]{HLWZ17}
G.~{Heier}, S.~S.~Y. {Lu}, B.~{Wong}, and F.~{Zheng}, \emph{{Reduction of
  manifolds with semi-negative holomorphic sectional curvature}}, ArXiv
  e-prints (2017).

\bibitem[KW74]{KW74}
Jerry~L. Kazdan and F.~W. Warner, \emph{Curvature functions for compact
  {$2$}-manifolds}, Ann. of Math. (2) \textbf{99} (1974), 14--47. \MR{0343205}

\bibitem[{Nom}16]{Nom16}
R.~{Nomura}, \emph{{K{\"a}hler manifolds with negative holomorphic sectional
  curvature, K{\"a}hler-Ricci flow approach}}, ArXiv e-prints (2016).

\bibitem[Roy80]{Roy80}
H.~L. Royden, \emph{The {A}hlfors-{S}chwarz lemma in several complex
  variables}, Comment. Math. Helv. \textbf{55} (1980), no.~4, 547--558.
  \MR{604712}

\bibitem[Tak08]{Tak08}
S.~Takayama, \emph{On the uniruledness of stable base loci}, J. Differential
  Geom. \textbf{78} (2008), no.~3, 521--541. \MR{2396253}

\bibitem[TY17]{TY15}
Valentino Tosatti and Xiaokui Yang, \emph{An extension of a theorem of
  {W}u--{Y}au}, J. Differential Geom. \textbf{107} (2017), no.~3, 573--579.
  \MR{3715350}

\bibitem[WWY12]{WWY12}
Pit-Mann Wong, Damin Wu, and Shing-Tung Yau, \emph{Picard number, holomorphic
  sectional curvature, and ampleness}, Proc. Amer. Math. Soc. \textbf{140}
  (2012), no.~2, 621--626. \MR{2846331}

\bibitem[WY16a]{WY16}
Damin Wu and Shing-Tung Yau, \emph{Negative holomorphic curvature and positive
  canonical bundle}, Invent. Math. \textbf{204} (2016), no.~2, 595--604.
  \MR{3489705}

\bibitem[WY16b]{WY16b}
\bysame, \emph{A remark on our paper ``{N}egative holomorphic curvature and
  positive canonical bundle''}, Comm. Anal. Geom. \textbf{24} (2016), no.~4,
  901--912. \MR{3570421}

\bibitem[Yau78]{Yau78}
Shing~Tung Yau, \emph{On the {R}icci curvature of a compact {K}\"ahler manifold
  and the complex {M}onge-{A}mp\`ere equation. {I}}, Comm. Pure Appl. Math.
  \textbf{31} (1978), no.~3, 339--411. \MR{480350}

\end{thebibliography}

\end{document}